\documentclass[12pt, a4paper]{article}

\usepackage{a4wide}
\usepackage[utf8]{inputenc}
\usepackage[english]{babel}
\usepackage{amsthm}
\usepackage{amsmath}
\usepackage{amssymb}
\usepackage{url}
\usepackage{enumerate}

\theoremstyle{plain}
\newtheorem{thm}{Theorem}
\newtheorem{lemma}{Lemma}

\theoremstyle{definition}

\theoremstyle{remark}

\newcommand{\F}{\mathcal{F}}
\newcommand{\N}{\mathbb{N}}
\renewcommand{\S}{\mathcal{S}}
\renewcommand{\P}{\mathbb{P}}

\begin{document}
\title{On the typical values of the cross-correlation measure}
\author{L\'aszl\'o M\'erai \\
{\small Johann Radon Institute for Computational and Applied Mathematics}\\
{ \small Austrian Academy of Sciences}\\
{ \small Altenbergerstr.\ 69, 4040 Linz, Austria}\\
{ \small e-mail: \texttt{merai@cs.elte.hu} }}
\maketitle

\begin{abstract}
Gyarmati, Mauduit and S\'ark\"ozy introduced the \textit{cross-correlation measure} $\Phi_k(\F)$ to measure the randomness of families of binary sequences $\F \subset \{-1,1\}^N$.

In this paper we study the order of magnitude of the cross-correlation measure $\Phi_k(\F)$ for typical families.
We prove that, for most families $\F \subset \{-1,1\}^N$ of size $2\leq |\F|<2^{N/12}$, $\Phi_k(\F)$ is of order $\sqrt{N\log \binom{N}{k}+k\log |\F|}$ for any given $2\leq k \leq N/(6\log_2 |\F|)$.
\end{abstract}


\textit{Key words and phases:} pseudorandom, binary sequence,correlation measure, cross-correlation measure
\let\thefootnote\relax\footnote{The final publication is available at Springer via \url{http://dx.doi.org/10.1007/s00605-016-0886-0}}

\section{Introduction}

Recently, in a series of papers the pseudorandomness of \textit{finite binary sequences} $E_N=(e_1,e_2,\dots, e_N)\in\{-1,1\}^N$ has been studied. In particular measures of pseudorandomness have been defined and investigated; see \cite{AKMMR,CMS,measures-survay,measures} and the references therein.

For example, Mauduit and S\'ark\"ozy \cite{measures} introduced the \textit{correlation measure of order $k$} $C_k(E_N)$ of the sequences $E_N$. Namely, for a $k$-tuple $D=(d_1,d_2,\dots, d_k)$ with non-negative integers $0\leq d_1 < d_2<\dots <d_k<N$ and $M\in\mathbb{N}$ with $M +d_k\leq N$ write
\begin{equation*}
 V_k(E_N,M,D)=\sum_{n=1}^M e_{n+d_1}e_{n+d_2}\dots e_{n+d_{k}}.
\end{equation*}
Then $C_k(E_N)$ is defined as
\begin{equation*}
 C_{k}(E_N)=\max_{M,D}\left| V(E_N,M,D) \right| =\max_{M,D}\left|
\sum_{n=1}^M e_{n+d_1}e_{n+d_2}\dots e_{n+d_{k}} \right|.
\end{equation*}

Cassaigne,  Mauduit and  S\'ark\"ozy \cite{CMS} studied the typical values of $C_k(E_N)$, when the binary sequences $E_N$ are chosen equiprobably from $\{-1,1\}^N$. Later Alon, Kohayakawa,  Mauduit,  Moreira and  R\"odl \cite{AKMMR} improved their result.

\begin{thm}\label{thm:AKMMR}
 For fixed $0<\varepsilon_0\leq 1/16$, there is a constant $N_0=N_0(\varepsilon_0)$ such that if $N\geq N_0$, then, with probability at least $1-\varepsilon_0$, we have
 \begin{align*}
  \frac{2}{5}\sqrt{N \log \binom{N}{k}}< C_k(E_N)&<\sqrt{(2+\varepsilon_1)N \log \left(N \binom{N}{k}\right)}\\
  & <\sqrt{(3+\varepsilon_0)N \log  \binom{N}{k}}<\frac{7}{4}\sqrt{N \log  \binom{N}{k}}.
 \end{align*}
 for every integer $k$ with $2\leq k \leq N/4$, where $\varepsilon_1=\varepsilon_1(N)=(\log \log N)/\log N$. 
\end{thm}

Recently, Schmidt \cite{schmidt} showed that for fixed $k$, the correlation measure $C_k$ of order $k$  converges strongly, and so has limiting distribution.

\bigskip

In order to study the pseudorandomness of \textit{families of finite binary sequences} $\F \subset \{-1,1\}^N$, Gyarmati, Mauduit and S\'ark\"ozy \cite{cross-correlation} introduced the notion of the \textit{cross-cor\-re\-la\-tion measure}.

Let $N,k\in\N$, and for any $k$ binary sequences $E_N^{(1)},E_N^{(2)}, \dots , E_N^{(k)}$ with
\[
 E_N^{(i)}=(e_1^{(i)},e_2^{(i)},\dots, e_N^{(i)})\in\{-1,1\}^N, \quad \text{for $i=1,2,\dots,k$},
\]
and any $M\in \N$ and $k$-tuple $D=(d_1,\dots,d_k)$ of non-negative integers with
\begin{equation}\label{cond:M,D}
 0\leq d_1\leq d_2\leq \dots \leq d_k< M+d_k\leq N, 
\end{equation}
write
\[
 V_k\left( E_N^{(1)},E_N^{(2)}, \dots , E_N^{(k)},M,D\right) =\sum_{n=1}^{M} e_{n+d_1}^{(1)}e_{n+d_2}^{(2)} \dots e_{n+d_k}^{(k)}.
\]

Let
\[
 \widetilde{C}_k\left( E_N^{(1)},E_N^{(2)}, \dots , E_N^{(k)}\right)=\max_{M,D}\left| V_k\left( E_N^{(1)},E_N^{(2)}, \dots , E_N^{(k)},M,D\right)\right|,
\]
where the maximum is taken over all $D=(d_1,\dots,d_k)$ and $M\in\N$ satisfying (\ref{cond:M,D}) with the additional restriction that if $E_N^{(i)}=E_N^{(j)}$ for some $i\neq j$, then we must not have $d_i=d_j$. Then the \textit{cross-correlation measure of order $k$} of the family $\F$ of binary sequences $E_N\in\{-1,1\}^N$ is defined as
\begin{equation*}
 \Phi_k (\F)=\max \widetilde{C}_k\left( E_N^{(1)},E_N^{(2)}, \dots , E_N^{(k)}\right),
\end{equation*}
where the maximum is taken over all $k$-tuples of binary sequences  

\[
\left( E_N^{(1)},E_N^{(2)}, \dots , E_N^{(k)}\right), \quad E_N^{(i)}\in\F, \quad \text{for } i=1, \dots, k.
\]

Clearly, for family $\F=\{E_N\}$ of size 1 we have
\[
 \Phi_k (\{E_N\})=C_k(E_N).
\]
On the other hand for general $\F$ we have
\[
  \Phi_k (\F)\geq \max_{E_N\in\F} C_k(E_N).
\]

\section{Typical values of $\Phi_k (\F)$}

In this paper we estimate $\Phi_k (\F)$ for "random" families $\F$ of sequences $E_N$ with given length $N$ and family size $|\F|$, i.e. we choose a family $\F$ from all subsets of $\{-1,1\}^N$ of size $|\F|$ with the same probability.

Clearly, the typical value of $\Phi_k (\F)$ strongly depends on the size of the family $\F$. If $\F$ is large: $|\F|>2^{cN}$ with some $0<c<1/2$, then $\Phi_k (\F)\gg N$ ($c=0.18$ can be chosen, see \cite{cross-correlation}). On the other hand, if $|\F|<2^{cN}$ with $c\leq 1/12 =0.0833\dots$, then the behavior of $\Phi_k (\F)$ can be controlled.

\begin{thm}\label{thm1}
 For a given $\varepsilon >0$, there exists $N_0$, such that if $N>N_0$ and $ 1\leq \log_2 |\F|< N/12$, then we have with probability at least $1-\varepsilon$, that
 \[
  \frac{2}{5}\sqrt{N\left(\log \binom{N}{k}+k\log|\F|\right)} <\Phi_k (\F)<\frac{5}{2} \sqrt{N\left(\log \binom{N}{k}+k\log|\F|\right)}
 \]
for every integer $k$ with $2\leq k \leq N/(6\log_2|\F|)$.
\end{thm}

The cross-correlation measure $\Phi$ can be also defined for \textit{binary sequence generators} instead of families of sequences. Namely, let $\S$ be a given set (set of parameters or seeds) and  $N\in \N$ be an integer. A binary sequence generator is a map $G: \S \rightarrow \{-1,1\}^N$ where
\[
  s \mapsto E_N(s)= (e_1(s), e_2(s),\dots, e_N(s))\in \{-1,1\}^N.
\]
For a survey of (pseudorandom) sequence generators, in particular their application in cryptography, see \cite[Chapters 5 and 6]{HAC}.

The cross-correlation measure of the generator $G$ can be defined in the following way:

Let $M$, $k_1, k_2, \dots, k_\ell\geq 1$ be integers with the restriction $k=k_1+k_2+\dots +k_\ell\geq 2$. Let $D=(d_1^1, d_2^1, \dots, d_{k_1}^1, d_1^2, d_2^2, \dots, d_{k_2}^2, \dots, d_1^\ell, d_2^\ell, \dots, d_{k_\ell}^\ell)$ be a $k$-tuple such that 

\begin{equation}\label{cond:D2}
 0\leq d_1^i< d_2^i< \dots< d_{k_i}^i<M+d_{k_i}^i\leq N, \quad \
\text{for } i=1, \dots, \ell.
\end{equation}
Then for distinct $s_1, s_2,\dots, s_\ell \in \S$ write
\begin{align*}
     &V_{k_1, k_2, \dots, k_\ell}\left(E_N(s_1),E_N(s_2), \dots, E_N(s_\ell),M,D  \right)\\ 
     &= \sum_{n=1}^{M} e_{n+d_1^1}(s_1)e_{n+d_2^1}(s_1)\dots e_{n+d_{k_1}^1}(s_1) \dots e_{n+d_1^\ell}(s_\ell)e_{n+d_2^\ell}(s_\ell)\dots e_{n+d_{k_\ell}^\ell}(s_\ell).
\end{align*}

The \textit{cross-correlation measure of order $k$} of the generator $G$ is defined as
\[
 \widetilde{\Phi}_k(G)=\max \left| V_{k_1, k_2, \dots, k_\ell}\left(E_N(s_1),E_N(s_2), \dots, E_N(s_\ell),M,D  \right) \right|,
\]
where the maximum is taken over all integers $k_1, k_2, \dots, k_\ell\geq 1$ such that $k=k_1+k_2+\dots +k_\ell$, all $s_1, s_2,\dots, s_\ell \in \S$, and all $M$ and $D$ satisfying (\ref{cond:D2}).

If the generator $G$ is collision free (injection), then $\widetilde{\Phi}_k(G)=\Phi_k (\F)$ with the family
\[
 \F=\F(G)=\{E_N(s): s\in \S\}.
\]
On the other hand, if there is a collision: $E_N(s)=E_N(s')$ for $s\neq s'$, then $\widetilde{\Phi}_k(G)=N$.

First, we estimate the value of $\widetilde{\Phi}_k(G)$ for "random" generator $G$. For each $s\in\S$ and $1\leq n\leq N$  we choose $e_n(s)$ from $\{-1,1\}$ uniformly and independently. Then we have

\begin{thm}\label{thm2}
For a given $\varepsilon >0$, there exists $N_0$, such that if $N>N_0$ and $1\leq \log_2|\S|< N/12$ 
then we have with probability at least $1-\varepsilon$, that
\begin{align*}
 \frac{2}{5} \sqrt{N\left(\log \binom{N}{k}+k\log|\S|\right)} <\widetilde{\Phi}_k (G) &< \frac{5}{2}\sqrt{N\left(\log \binom{N}{k}+k\log|\S|\right)}
\end{align*}
for every integer $k$ with $2\leq k \leq N/(6\log_2|\S|)$.
\end{thm}

We can prove Theorem \ref{thm1} as a corollary of Theorem \ref{thm2}.

\begin{proof}[Theorem \ref{thm1}]
Throughout the proof we assume, that the integer $N$ is large enough.

First we show, that for $|\S|<2^{cN}$ with $0<c<1/2$, the probability of the collision is small:
\begin{equation}\label{eq:collision_prob}
 \P(\not \exists s,s'\in \S: \ E_N(s)=E_N(s')) = 1-o(1) 
\end{equation}

Indeed, this probability is
\begin{align*}
 &\frac{\binom{2^N}{|\S|}\cdot|\S|!}{\left(2^N \right)^{|\S|}}=\left(1-\frac{1}{2^N}\right)\cdot \left(1-\frac{2}{2^N}\right)\dots \left(1-\frac{|\S|-1}{2^N}\right) \geq \left(1-\frac{|\S|}{2^N} \right)^{|\S|}.
\end{align*}
Since for all $0<\delta<1$ there is $N_0$ such that if $N\geq N_0$ we have
\[
 \left(1-\frac{|\S|}{2^N} \right)^{|\S|}\geq \left(1-\frac{1}{2^N/|\S|} \right)^{\delta 2^N/|\S|}\geq \left(e^{-\delta}+o(1)\right),
\]
which gives  (\ref{eq:collision_prob}).

Now let us assume, that Theorem \ref{thm2} holds with $\varepsilon_1$ and let $\varepsilon'$ be the probability of the collision. Then for a random generator $G$ we have
\begin{align*}
 &\varepsilon_1 > \P\left(\widetilde{\Phi}_k(G)>   \frac{5}{2} \sqrt{N\left(\log \binom{N}{k}+k\log|\S|\right)}\right)\\ 
 &  = \P\left(\widetilde{\Phi}_k(G)>   \left. \frac{5}{2} \sqrt{N\left(\log \binom{N}{k}+k\log|\S|\right)} \ \right| \text{there is no collision}\right) \\
 &\qquad \cdot \P(\text{there is no collision})\\
 &\quad + \P\left(\widetilde{\Phi}_k(G)>   \left.\frac{5}{2} \sqrt{N\left(\log \binom{N}{k}+k\log|\S|\right)} \ \right| \text{there is a collision}\right) \\
 &\qquad \cdot \P(\text{there is a collision})\\
 &= \P\left(\widetilde{\Phi}_k(G)>   \left. \frac{5}{2} \sqrt{N\left(\log \binom{N}{k}+k\log|\S|\right)} \ \right| \text{there is no collision}\right) (1-\varepsilon')\\
 & \quad + 1\cdot \varepsilon'.
\end{align*}
If $G$ is chosen uniformly from all generators with the condition that there is no collision, then the family $\F=\F(G)$ is uniformly distributed within  all families of size $|\F|=|\S|$. Thus
\begin{align*}
&\P\left(\widetilde{\Phi}_k(G)>   \left. \frac{5}{2} \sqrt{N\left(\log \binom{N}{k}+k\log|\S|\right)} \ \right| \text{there is no collision}\right) \\
&=  \P\left(\Phi_k(\F(G))>   \frac{5}{2} \sqrt{N\left(\log \binom{N}{k}+k\log|\F(G)|\right)}\right)
\end{align*}
and so
\begin{equation*}
 \P\left(\Phi_k(\F)>   \frac{5}{2} \sqrt{N\left(\log \binom{N}{k}+k\log|\F|\right)}\right)<\frac{\varepsilon_1-\varepsilon'}{1-\varepsilon'}.
\end{equation*}

In the same way we get
\begin{equation*}
 \P\left(\Phi_k(\F)<   \frac{2}{5} \sqrt{N\left(\log \binom{N}{k}+k\log|\F|\right)}\right)>1-\frac{\varepsilon_1-\varepsilon'}{1-\varepsilon'}.
\end{equation*}

Choosing $\varepsilon=\frac{\varepsilon_1-\varepsilon'}{1-\varepsilon'}$  we get the result.
\end{proof}

\section{Estimates for $\widetilde{\Phi}_k (G)$ for random generator $G$}

In this section we consider $G$ as a "random" generator i.e. $e_n(s)$ are independent and uniform random variables in $\{-1,1\}$, for each $s\in\S$ and $1\leq n\leq N$.

\subsection{Estimates for the binomial distribution}
The proof of Theorem \ref{thm2} is based on estimations on tails of the binomial distribution.
First we summarize some basic facts about their properties. 

Let $S(n,p)$ be the sum of $n$ independent Bernoulli random variables with mean $p$. First we state the following consequences of the de Moivre-Laplace theorem (see e.g. \cite[Chapter 1, Theorem 6]{Bollobas}) for $p=1/2$.

\begin{lemma}\label{lemma:binom}
\begin{enumerate}[(i)] 
 \item  \label{lemma:binom-1} For any $c=c(n)>0$ with $c=o(n^{1/6})$, we have
 \begin{align}\label{eq:ML1a}
  \P\left(S(n,1/2)\geq \left\lfloor \frac{n}{2}\right\rfloor +c\sqrt{n} \right)
  &=\sum _{\ell\geq c\sqrt{n}}\frac{1}{2^n}\binom{n}{\lfloor n/2 \rfloor+\ell}\notag \\ 
  &=\left(\sqrt{\frac{2}{\pi}+o(1)} \right)\left(\int_c^\infty e^{-2x^2}dx\right).
 \end{align}
In particular, if we further have that $c \rightarrow \infty$, then
 \begin{align}\label{eq:ML1b}
  \P\left(S(n,1/2)\geq \left\lfloor \frac{n}{2}\right\rfloor +c\sqrt{n} \right)
  &=\frac{e^{-2c^2}}{2c\sqrt{2 \pi}}(1+o(1)).
 \end{align}

 \item \label{lemma:binom-2} The estimates (\ref{eq:ML1a}) and (\ref{eq:ML1b}) hold for the lower tail
 \[
  \P\left(S(n,1/2)\leq \left\lfloor \frac{n}{2}\right\rfloor -c\sqrt{n} \right)
 \]
as well.
\end{enumerate}
 \end{lemma}

Let $\{x\}=x-\lfloor x\rfloor$. We have the following lower estimate for the symmetric binomial distribution (see \cite[Fact 10]{AKMMR}).

\begin{lemma}\label{lemma:binom=}
 Let $n$ and $c$ be integers with
 \begin{equation*}
  -\left\lfloor \frac{n}{2} \right\rfloor \leq c \leq \left\lceil  \frac{n}{2} \right\rceil.
 \end{equation*}
If $n$ is sufficiently large, then
 \begin{align*}
  \P\left(S(n,1/2)= \left\lfloor \frac{n}{2}\right\rfloor +c \right)
  =\frac{1}{2^n}\binom{n}{\lfloor n/2 \rfloor+c}
  \geq (1+o(1))2^{-4(c+\{n/2\})^2/n}\sqrt{\frac{2}{\pi n}}.
 \end{align*}
\end{lemma}

Let
\[
 S^{\pm}(n)=\sum_{1\leq i\leq n} X_i,
\]
where $X_i$ ($1\leq i \leq n$) are independent random variables with mean 0, that is,
\[
 \P(X_i=-1)=\P(X_i=+1)=1/2.
\]
Clearly, $(S^{\pm}(n)+n)/2$ is binomially distributed with parameters $n$ and $1/2$. The following lemma states a well-known estimate for large deviation of $S^{\pm}(n)$ (see e.g. \cite[Appendix 2]{AlonSpencer}).

\begin{lemma}\label{lemma:Spm}
 Let $X_i$ ($1\leq i \leq n$) be independent $\pm 1$ random variables with mean 0. Let  $S^{\pm}(n)=\sum_{1\leq i\leq n} X_i$. For any real number $a>0$, we have
 \[
  \P(S^{\pm}(n) >a)<e^{-a^2/2n}.
 \]

\end{lemma}

\subsection{Proof of Theorem \ref{thm2}}

We prove Theorem \ref{thm2} in two parts. First, we prove the upper estimate for $\widetilde{\Phi}_k(G)$ for typical generator $G$.

\begin{lemma}\label{lemma:thm2-upper}
For $1\leq  \log_2|\S|< \log_2 N$  we have 
\[
 \widetilde{\Phi}_k (G)<2\sqrt{N \left(\log \binom{N}{k}+k\log |\S|\right)},
\]
and for  $\log_2 N\leq \log_2 |\S|<N/12$ we have 
\begin{align*}
 \widetilde{\Phi}_k (G)<2\sqrt{N \left(k\log N+\log \binom{|\S|}{k}\right)}
	  <2\sqrt{N \left(\log \binom{N}{k}+(1+o(1))k\log |\S|\right)}
\end{align*}
with probability tending to 1 as $N\rightarrow \infty$ for every integer $k$ with $2\leq k \leq  N/(6\log_2|\S|)$.
\end{lemma}

\begin{proof}
Assume first, that  $1\leq \log_2 |\S|<\log_2 N$.

Let us consider the event
\begin{equation}\label{eq:thm2-upper-1}
V_{k_1,k_2,\dots, k_\ell}\left(E_N(s_1),E_N(s_2), \dots E_N(s_\ell), M, D \right)> 2\sqrt{N\left(\log \binom{N}{k}+k\log |\S|\right)} 
\end{equation}
for fixed integers $k$, $\ell$, $k_1,k_2,\dots, k_\ell$, $M$ and $k$-tuple $D$ with restrictions $k=k_1+k_2+\dots+ k_\ell$ and (\ref{cond:D2}).

Since $e_n(s)$ are independent for $1\leq n \leq N$ and $s\in\S$, then 
\begin{equation*}
z_n=e_{n+d_1^1}(s_1)e_{n+d_2^1}(s_1)\dots e_{n+d_{k_1}^1}(s_1) \dots e_{n+d_1^\ell}(s_\ell)e_{n+d_2^\ell}(s_\ell)\dots e_{n+d_{k_\ell}^\ell}(s_\ell)
\end{equation*}
are also independent and uniform in $\{-1,1\}$. 
This follows from the observations, that
for each $j$ the sequence
\[
 \left(e_{1+d_1^j}(s_j)\cdots e_{1+d_{k_j}^j}(s_j), \dots,  e_{N-d_{k_j}^j+d_1^j}(s_j)\cdots e_{N}(s_j)\right)
\]
is uniformly distributed in $\{-1,1\}^{N-d_{k_j}^j}$,
and the sequence
\[
 (x_1(1),\dots, x_m(1),\dots ,x_1(p),\dots,x_m(p))
\]
is uniform in $\{-1,1\}^{pm}$ if and only if
\[
 (x_1(1),\dots, x_m(1), x_1(p-1),\dots,x_m(p-1), x_1(1)\cdots x_1(p),\dots, x_1(p)\cdots x_m(p))
\]
is uniform in $\{-1,1\}^{pm}$.
%
%
%

Then 
\[
V_{k_1,k_2,\dots, k_\ell}\left(E_N(s_1),E_N(s_2), \dots E_N(s_\ell), M, D \right)
\]
has the same distribution as $S^{\pm}(M, 1/2)$. By Lemma \ref{lemma:Spm} we have, that (\ref{eq:thm2-upper-1}) holds with probability less than
\begin{equation*}
 \exp \left\{-\frac{1}{2M} 4N\left(\log  \binom{N}{k} + k\log |\S| \right)\right\}\leq \left( \binom{N}{k} |\S|^{k}\right)^{-2}.
\end{equation*}

Summing over all possible choices of $\ell$, $k_1,k_2,\dots, k_\ell$, $s_1,s_2,\dots, s_\ell$, $M$ and $D$ we get

\begin{align}\label{eq:thm2-upper-3-S-kicsi}
 &\P\left(\widetilde{\Phi}_k (G) >2\sqrt{N \left(\log \binom{N}{k}+k\log |\S|\right)}\right) \notag \\
 &\qquad \leq \sum_{\ell} \sum_{k_1,k_2,\dots, k_\ell} \sum_{s_1,s_2,\dots, s_\ell} \sum_{M} \sum_{D} \left( \binom{N}{k} |\S|^k\right)^{-2}.
\end{align}

For $k\leq |\S|$ we estimate the number of $k$-tuples $D$ by $N^k$. Thus (\ref{eq:thm2-upper-3-S-kicsi}) is less than
\begin{align}\label{eq:thm2-upper-4-S-kicsi}
 &\quad \left( \binom{N}{k}|\S|^{k}\right)^{-2} \sum_{\ell=1}^{k}\binom{k-1}{\ell-1} \binom{|\S|}{\ell}  N^{k+1} \notag\\
 &\leq 2^{k-1} \frac{\binom{|\S|}{k} N^{k+1}}{\left( \binom{N}{k} |\S|^k\right)^{2}} \leq 2^{k-1} \frac{\frac{e^k|\S|^k}{k^k} N^{k+1}}{\frac{N^{2k}}{k^{2k}} |\S|^{2k}}=\frac{2^{k-1}(ek)^k}{|\S|^k N^{k-1}}\leq e \left( \frac{2e}{N}\right)^{k-1}\notag\\
 &\leq \frac{2e^2}{N}\left( \frac{e}{3}\right)^{k-2},
\end{align}
where we used $\left(\frac{a}{b}\right)^b\leq \binom{a}{b}\leq \left(\frac{ea}{b}\right)^b$.

Next, consider (\ref{eq:thm2-upper-3-S-kicsi}) for $k> |\S|$. We estimate the number of $k$-tuples $D$ of form (\ref{cond:D2}) with the restriction $\ell\leq |\S|$ by 
\begin{align}\label{eq:thm2-upper-6-S-kicsi}
 \binom{N}{k_1}\binom{N}{k_2}\dots \binom{N}{k_\ell}\leq \frac{(eN)^{k_1}}{k_1^{k_1}}\frac{(eN)^{k_2}}{k_2^{k_2}}\dots \frac{(eN)^{k_\ell}}{k_\ell^{k_\ell}}
 =\frac{(eN)^{k}}{e^{k_1\log k_1+k_2 \log k_2+\dots +k_\ell\log k_\ell}}
\end{align}
Since the function $x \log x$ (with $0 \log 0 =0$) is convex, writing $k_{\ell+1}=\dots=k_{|\S|}=0$, we get by the Jensen inequality, that
\begin{equation*}
 k_1\log k_1+k_2 \log k_2+\dots +k_\ell\log k_\ell=\sum_{i=1}^{|\S|}k_i \log k_i \geq k \log \frac{k}{|\S|}.
\end{equation*}
Whence we get that (\ref{eq:thm2-upper-6-S-kicsi}) is less than
\begin{equation}\label{eq:thm2-upper-7-S-kicsi}
 \frac{(eN)^{k}}{\left(\frac{k}{|\S|}\right)^k}\leq \binom{N}{k} (e|\S|)^k.
\end{equation}
By (\ref{eq:thm2-upper-6-S-kicsi}) and (\ref{eq:thm2-upper-7-S-kicsi}) we have that (\ref{eq:thm2-upper-3-S-kicsi}) for $k> |\S|$ is less than
\begin{align}\label{eq:thm2-upper-8-S-kicsi}
 &\left( \binom{N}{k} |\S|^k\right)^{-2} \binom{N}{k} (e|\S|)^k\sum_{\ell} \sum_{k_1,k_2,\dots, k_\ell} \sum_{s_1,s_2,\dots, s_\ell} \sum_{M} 1 \notag \\
 &\leq  \left( \binom{N}{k} |\S|^k\right)^{-1} e^k N \sum_{\ell =1}^{|\S|} \binom{k-1}{\ell -1} |\S|^\ell \notag \\
 &\leq  \left( \binom{N}{k} |\S|^k\right)^{-1} e^k N|\S| \sum_{\ell =0}^{k-1} \binom{k-1}{\ell } |\S|^\ell \notag \\
 & =\left( \binom{N}{k} |\S|^k\right)^{-1} e^k N|\S| (|\S|+1)^{k-1}\notag \\
 & \leq e k \left(\frac{2ek}{N} \right)^{k-1}\leq \frac{2(ek)^2}{N} \left(\frac{e}{3} \right)^{k-2}.
\end{align}
Finally, by (\ref{eq:thm2-upper-3-S-kicsi}), (\ref{eq:thm2-upper-4-S-kicsi}) and (\ref{eq:thm2-upper-8-S-kicsi}) we get, that for a fixed $k$, the probability of
\begin{align}\label{eq:prob-S-kicsi}
 \widetilde{\Phi}_k (G) >2\sqrt{N \left(\log \binom{N}{k}+k\log |\S|\right)}
\end{align}
is
\begin{equation*}
 O\left(\frac{1}{N}  k^2 \left(\frac{e}{3} \right)^{k-2} \right).
\end{equation*}
Summing it for $2\leq k \leq  N/(6\log_2|\S|)$ we get that the probability that (\ref{eq:prob-S-kicsi}) holds for some $k$ is
\begin{align*}
 O\left( \sum_{2\leq k \leq  N/(6\log_2|\S|)}\frac{1}{N}  k^2\left(\frac{e}{3} \right)^{k-2} \right) = O\left( \frac{1}{N} \sum_{k=0}^{\infty}k^2\left(\frac{e}{3} \right)^{k} \right)= O\left( \frac{1}{N}\right). 
\end{align*}

Now suppose that $\log_2 N\leq \log_2|\S|<N/12$. One may get in the same way, that
\begin{align}\label{eq:thm2-upper-3}
 &\P\left(\widetilde{\Phi}_k (G) >2\sqrt{N \left(k\log N+\log \binom{|\S|}{k}\right)}\right) \notag \\
 &\qquad \leq \sum_{\ell} \sum_{k_1,k_2,\dots, k_\ell} \sum_{s_1,s_2,\dots, s_\ell} \sum_{M} \sum_{D} \left( N ^k\binom{|\S|}{k}\right)^{-2}.
\end{align}

Estimating trivially the number of terms, we get that (\ref{eq:thm2-upper-3}) is less than
\begin{align*}
 \left( N ^k\binom{|\S|}{k}\right)^{-2} \sum_{\ell=1}^{k}\binom{k-1}{\ell-1} \binom{|\S|}{\ell}  N^{k+1} \leq 
 \left( N ^k\binom{|\S|}{k}\right)^{-2} 2^{k-1}\binom{|\S|}{k} N^{k+1}
 \leq  \frac{2^{k-1}}{N^{k-1}\binom{|\S|}{k}}.
\end{align*}

Summing over $2\leq k \leq  N/(6\log_2|\S|)$ 
we get, that the probability of (\ref{eq:thm2-upper-1}) for some $k$ is less than
\begin{align*}
 N \sum_{2\leq k \leq  N/(6\log_2|\S|)} \frac{2^k}{\binom{|\S|}{k} N^{k}}
<\frac{1}{N } \sum_{k=0}^{\infty} \frac{1}{N^k} =O\left(\frac{1}{N}\right)
\end{align*}
which gives the result.
\end{proof}

Next, we prove the lower estimate for $\widetilde{\Phi}_k(G)$ for typical generator $G$.

\begin{lemma}\label{lemma:thm2-lower}
Let $m=\lfloor N/3 \rfloor$.
For $1\leq \log_2|\S|\leq m^{1/4}$ we have 
\[
 \widetilde{\Phi}_k (G)>\frac{4}{9} \sqrt{N\left(\log \binom{N}{k} +k \log |\S|\right)},
\]
and for $m^{1/4}<\log_2|\S|<N/12$ we have 
\begin{align*}
 \widetilde{\Phi}_k (G)&>\frac{4}{9} \sqrt{N\left(k\log N + \log \binom{|\S|}{k}\right)}\\
                       &>\frac{4}{9} \sqrt{N\left(\log \binom{N}{k} +(1-o(1))k \log |\S|\right)} ,
\end{align*}
with probability tending to 1 as $N\rightarrow \infty$ for every integer $k$ with $2\leq k \leq N/(6\log_2|\S|)$.
\end{lemma}

We start with the following form of Fact 16 in \cite{AKMMR}.

\begin{lemma}\label{lemma:log-binom}
 Let $m=\lfloor N/3\rfloor$. For every sufficiently large $N$, the followings hold.
 \begin{enumerate}[(i)]
  \item If $2\leq k\leq \log m$, then
  \begin{equation*}
   \log \binom{N/3}{k} \geq 0.98 \log \binom{N}{k}.
  \end{equation*}
  \item If $\log m< k\leq N/(6\log_2 |\S|)$, then
  \begin{equation*}
   \log \binom{N/3}{k}\geq \frac{1-10^{-10}}{3}\log \binom{N}{k}.
  \end{equation*}
 \end{enumerate}

\end{lemma}

Let $m=\lfloor N/3 \rfloor$ and for $1\leq \log_2 |\S|\leq m^{1/4}$
consider the maximal $r=r_k(m,\S)\in \N$ such that
\begin{equation*}
 \P\left(S(m,1/2)\geq \frac{1}{2}(m+r) \right)\geq \frac{k^2 \log N}{\binom{m+1}{k-1}|\S|^{k} }
\end{equation*}
holds, and for $m^{1/4}<\log_2 |\S|\leq N/12$ consider the maximal $r=r_k(m,\S)\in \N$ such that
\begin{equation*}
 \P\left(S(m,1/2)\geq \frac{1}{2}(m+r) \right)\geq \frac{k^2 \log N}{(m+1)^{k-1}\binom{|\S|}{k} }
\end{equation*}
holds.

We give a lower estimate to $r_k(m,\S)$ for large and small $\S$ separately.

\begin{lemma}\label{lemma:r}
For every sufficiently large $N$ and for $1\leq \log_2 |\S|\leq m^{1/4}$ the followings hold.
\begin{enumerate}[(i)]
 \item \label{lemma:r-kisk-r} For $2\leq k\leq \log m$ we have
\begin{equation*}
 r_k(m,\S)\geq 0.99 \sqrt{2m \left(\log\binom{m+1}{k-1} +k\log |\S|\right)}.
\end{equation*}

\item \label{lemma:r-nagyk-r} For $\log m<k\leq N/(6 \log_2 |\S|)$ we have
\begin{equation*}
 r_k(m,\S)\geq  (1-10^{-10})\sqrt{\frac{1}{\log 2} m \left(\log\binom{m+1}{k-1} +k\log |\S|\right)}.
\end{equation*}

\item \label{lemma:r-osszesk-r} For $2\leq k\leq  N/(6 \log_2 |\S|)$ we have
\begin{equation*}
 r_k(m,\S)\geq \frac{4}{9} \sqrt{N \left(\log\binom{N}{k} +k\log |\S|\right)}.
\end{equation*}
\end{enumerate}
\end{lemma}

\begin{proof}[Lemma \ref{lemma:r}]
First we remark that for all $2\leq k \leq N/(6 \log_2 |\S|)$, we have
\begin{equation*}
k^2\log N \leq \binom{m+1}{k-1}^{o(1)},
\end{equation*}
(see e.g. \cite{AKMMR}).

First assume, that $k\leq \log m$.
Let
\begin{equation*}
 r=\left\lceil 0.99 \sqrt{2m \left(\log\binom{m+1}{k-1} +k\log |\S|\right)}\right\rceil
\end{equation*}
and
\begin{equation*}
 c=\frac{r+1}{2\sqrt{m}}=(1+o(1))0.99 \sqrt{\frac{1}{4}  \left(\log\binom{m+1}{k-1} +k\log |\S|\right)}.
\end{equation*}
Since now $c=o\left(m^{1/6}\right)$, by (\textit{\ref{lemma:binom-1}})  of Lemma \ref{lemma:binom} we have
\begin{align*}
 &\P\left(S(m,1/2)\geq \frac{1}{2}(m+r) \right) \geq \P\left(S(m,1/2)\geq \left\lfloor \frac{m}{2}  \right\rfloor +c \sqrt{m} \right)\\
 =&\frac{e^{-2c^2}}{2c\sqrt{2\pi}}(1+o(1)) 
 \geq \frac{1}{4}\left(\binom{m+1}{k-1} |\S|^k\right)^{-0.99   } \geq \frac{k^2 \log N }{\binom{m+1}{k-1} |\S|^k}
\end{align*}
which proves (\textit{\ref{lemma:r-kisk-r}}).

To prove (\textit{\ref{lemma:binom-2}}) assume, that $\log m<k\leq N/(6 \log_2 |\S|)$.
Let
\begin{equation*}
 r=\left\lceil (1-10^{-10})\sqrt{\frac{1}{\log 2} m \left(\log\binom{m+1}{k-1} +k\log |\S|\right)}\right\rceil
\end{equation*}
and
\begin{equation*}
 c=\left\lceil\frac{r+1}{2} \right\rceil=(1+o(1))\frac{1-10^{-10}}{2 \sqrt{\log 2}} \sqrt{m\left( \log\binom{m+1}{k-1} +k\log |\S|\right)}.
\end{equation*}
Since now $0<c<m/2$, by Lemma \ref{lemma:binom=}   we have
\begin{align*}
 &\P\left(S(m,1/2)\geq \frac{1}{2}(m+r) \right) \geq \P\left(S(m,1/2)\geq \left\lfloor \frac{m}{2}  \right\rfloor +c  \right)\notag\\
 \geq & (1+o(1))2^{-4(c+1/2)^2/m}\sqrt{\frac{2}{\pi m}} 
 \geq  (1+o(1))2^{-r^2/m} 2^{-(6r+9)/m}\sqrt{\frac{2}{\pi m}}\notag\\
 \geq &\left(\binom{m+1}{k-1}|\S|^k\right)^{-1+10^{-10}} 2^{-(6r+9)/m}\sqrt{\frac{2}{\pi m}}\notag\\
 \geq & (1+o(1))\left(\binom{m+1}{k-1}|\S|^k\right)^{-1+10^{-10}}\geq \frac{k^2 \log N}{\binom{m+1}{k-1}|\S|^k}.
\end{align*}

Finally, (\textit{\ref{lemma:r-osszesk-r}}) follows from (\textit{\ref{lemma:r-kisk-r}}), (\textit{\ref{lemma:r-nagyk-r}}) and Lemma \ref{lemma:log-binom} in the same way as in \cite{AKMMR}. Namely, if $2\leq k \leq \log m$, then 
\[
 \binom{m+1}{k-1}\geq \binom{N/3}{k-1}\geq \binom{N/3}{k}^{1/2},
\]
thus 
\begin{align*}
 r_k(m,\S)&\geq 0.99 \sqrt{2\left \lfloor \frac{N}{3} \right\rfloor \left(\log\binom{m+1}{k-1} +k\log |\S|\right)}\\
     &\geq (1+o(1))\frac{0.99}{\sqrt{3}}\sqrt{ N \left(\log\binom{N/3}{k} +k\log |\S|\right)}\\
     &\geq \frac{4}{9}\sqrt{ N \left(\log\binom{N}{k} +k\log |\S|\right)}.
\end{align*}

On the other hand, if $\log m < k \leq N/(6 \log_2 |\S|)$, then
\[
 \binom{m+1}{k-1}\geq\binom{N/3}{k}^{1-o(1)},
\]
thus
\begin{align*}
 r_k(m,\S)& \geq \frac{1-10^{-10}}{\sqrt{\log 2}}\sqrt{\left\lfloor\frac{N}{3}\right\rfloor\left(\log\binom{m+1}{k-1} +k\log |\S|\right)}\\
 & \geq (1+o(1))\frac{1-10^{-10}}{\sqrt{3\log 2}}\sqrt{N\left(\log\binom{N}{k} +k\log |\S|\right)}\\
 & \geq \frac{4}{9}\sqrt{N\left(\log\binom{N}{k} +k\log |\S|\right)}.
\end{align*}
\end{proof}

The lower estimate to $r_k(m,\S)$ for small $\S$ can be prove similarly.

\begin{lemma}\label{lemma:r-2}
For every sufficiently large $N$ and for $m^{1/4}<\log_2 |\S|<N/12$
\begin{equation*}
 r_k(m,\S)\geq \frac{4}{9} \sqrt{N \left(\log\binom{N}{k} +k\log |\S|\right)}.
\end{equation*}
holds for $2\leq k\leq  N/(6 \log_2 |\S|)$.
\end{lemma}

We also need the following lemma (\cite[Lemma 19]{AKMMR}).

\begin{lemma}\label{lemma:unio}
 Let $A_1, A_2, \dots , A_M$ be events in a probability space, each with probability at least $p$. Let $\varepsilon \geq 0$ be given, and suppose that
 \begin{equation*}
  \P(A_i \cap A_j)\leq p^2(1+\varepsilon)
 \end{equation*}
for all $i\neq j$. Then
\begin{align*}
 \P\left(\bigcup_{i=1}^M A_i \right)\geq 1-\varepsilon-\frac{2}{Mp}.
\end{align*}

\end{lemma}

Now we are in the state to prove Lemma \ref{lemma:thm2-lower}.

\begin{proof}[Lemma \ref{lemma:thm2-lower}] 
First we remark, that it is enough  to show that
\begin{equation}\label{eq:cel}
 \widetilde{\Phi}_k(G)\leq r_k(m,\S)
\end{equation}
holds with probability at most $ O(1/k^2 \log N)$. 

Indeed, summing over all $2\leq k \leq N/(6\log_2|\S|)$ we get that (\ref{eq:cel}) holds for \textit{some} $k$ with $2\leq k \leq N/(6\log_2|\S|)$ with probability $O(1/\log N)=o(1)$. Whence (\ref{eq:cel}) does not hold for all $2\leq k \leq N/(6\log_2|\S|)$ with probability $1-o(1)$, which proves the lemma.

We prove the lemma for small $\S$, for large $\S$ one can obtain the result in the same way referring to Lemma \ref{lemma:r-2} instead of Lemma \ref{lemma:r}. So assume, that $1\leq \log_2|\S|\leq m^{1/4}$.

For $s_1\in\S$ let
 \begin{equation*}
  v(s_1)=(e_1(s_1), e_2(s_1), \dots, e_m(s_1))
 \end{equation*}
and for $2\leq \ell \leq k$, for $s_2, \dots, s_\ell\in\S$, for $(k_2,k_3,\dots, k_\ell)$ with $k_2+k_3+\dots +k_\ell=k-1$ and for 
$D=(0,d^{2}_1,d^{2}_2,\dots, d^{2}_{k_2},\dots,d^{\ell}_1,d^{\ell}_2,\dots, d^{\ell}_{k_\ell})$ with $m\leq d^{j}_1<\dots< d^{j}_{k_\ell}\leq 2m$ for $j=2,3,\dots, \ell$ let
 \begin{align*}
  v_\ell(s_2,s_3,\dots,s_\ell, D )=\left(\prod_{t=2}^{\ell} e_{1+d^{t}_1}(s_t)\cdots e_{1+d^{t}_{k_t}}(s_t),\dots, \prod_{t=2}^{\ell} e_{m+d^{t}_1}(s_t)\cdots e_{m+d^{t}_{k_t}}(s_t) \right)\\
 \end{align*}
 
Let $A_\ell(s_1,s_2,\dots, s_\ell,D)$ be the event
\begin{equation*}
 |\langle v(s_1), v_\ell(s_2,s_3,\dots,s_\ell, D ) \rangle|  \geq  r_k(m,\S).
\end{equation*}
Since $\langle v(s_1), v_\ell(s_2,s_3,\dots,s_\ell, D ) \rangle$ has the same distribution as $S(m,1/2)$,
we have
\[
 p=\P\left(A_\ell(s_1,s_2,\dots, s_\ell,D)\right)=2\cdot \P\left(S(m,1/2)\geq \frac{1}{2}(m+r_k(m,\S)) \right).
\]

One can obtain in the same way as \cite[Claim 18]{AKMMR}, that the events $A_\ell$ are pairwise independent.

\begin{lemma}\label{lemma:A-independent}
 For  $\{s_1,s_2,\dots, s_\ell\}\neq \{s'_1,s'_2,\dots, s'_{\ell '}\}$ or $D\neq D'$ we have
 \[
  \P(A_\ell (s_1,s_2,\dots, s_\ell,D) \cap A_{\ell'}(s'_1,s'_2,\dots, s'_{\ell'},D'))=p^2.
 \]
\end{lemma}

Let $\mathcal{D}_k(\S)$ be the number of possible $\ell$, $s_1,s_2,\dots, s_\ell$ and $D$, then
by Lemmas \ref{lemma:unio} and \ref{lemma:A-independent} we get that
\begin{align}
 &\P\left( \widetilde{\Phi}_k(G)\geq r_k(m,\S) \right) \notag\\
 &\geq \P \left(\bigcup_{\ell=2}^k \bigcup_{\substack{k_2,\dots, k_\ell\geq 1\\k_2,\dots, k_\ell =k-1}} \bigcup_{s_1,s_2,\dots, s_\ell}\bigcup_{D}A_\ell(s_1,s_2,\dots, s_{k_\ell},D) 
 \right)\geq 1-\frac{2}{p\cdot \mathcal{D}_k(\S)}.\label{eq:majdnemvege}
\end{align}
Finally,  we give a lower bound to \eqref{eq:majdnemvege} for $|\S|<m$ and $|\S|\geq m$ separately.
If $|\S|<m$, then
\begin{align*}
 &p \cdot \mathcal{D}_k(\S)  =p \sum_{\ell=2}^{k}\sum_{\substack{k_2,\dots, k_\ell\geq 1\\k_2,\dots, k_\ell =k-1}}\sum_{s_1,s_2,\dots, s_\ell}\sum_{D} 1\notag \\
 =& p \sum_{\ell=2}^{k}\sum_{\substack{k_2,\dots, k_\ell\geq 1\\k_2,\dots, k_\ell =k-1}}\sum_{s_1,s_2,\dots, s_\ell} \binom{m+1}{k_2}\binom{m+1}{k_3}\dots \binom{m+1}{k_\ell}\notag \\
 \geq & p \sum_{\ell=2}^{k} \binom{k-2}{\ell-2}\binom{|\S|}{\ell}m^{\ell-2}\binom{m+1}{k-1}\notag \\
 \geq & p \binom{|\S|}{2} \binom{m+1}{k-1}\sum_{\ell=2}^{k} \binom{k-2}{\ell-2} |\S|^{\ell -2}\notag \\
 \geq &\frac{1}{4} p  \binom{m+1}{k-1} (|\S|)^{k}\geq \frac{1}{2}k^2 \log N
\end{align*}
Similarly, for $|\S|\geq m$ we have
\begin{align*}
 &p \cdot \mathcal{D}_k(\S)= p \sum_{\ell=2}^{k}\sum_{\substack{k_2,\dots, k_\ell\geq 1\\k_2,\dots, k_\ell =k-1}}\sum_{s_1,s_2,\dots, s_\ell}\sum_{D} 1\notag \\
 \geq& p \sum_{s_1,s_2,\dots, s_k} \binom{m+1}{1}\dots \binom{m+1}{1}
 = p \binom{|\S|}{k}(m+1)^{k-1}\geq \frac{1}{2} k^2 \log N
\end{align*}
which proves the result.
\end{proof}

\section*{Acknowledgements}

The author is partially supported by the Austrian Science Fund FWF Project F5511-N26 
which is part of the Special Research Program "Quasi-Monte Carlo Methods: Theory and Applications" and by Hungarian National Foundation for Scientific Research,
Grant No. K100291.

\end{document}